\documentclass[12pt]{article}
\usepackage{amsmath,enumerate,amsfonts,amssymb,color,graphicx,amsthm}

\usepackage[colorlinks=true, allcolors=black]{hyperref}
\usepackage{todonotes}
\usepackage[normalem]{ulem}
\numberwithin{equation}{section}
\usepackage{cite}
\usepackage{mathtools}
\usepackage[nottoc,notlot,notlof]{tocbibind}
\usepackage[toc,page]{appendix}

\usepackage{lipsum}

\newlength{\leftstackrelawd}
\newlength{\leftstackrelbwd}
\def\leftstackrel#1#2{\settowidth{\leftstackrelawd}%
	{${{}^{#1}}$}\settowidth{\leftstackrelbwd}{$#2$}%
	\addtolength{\leftstackrelawd}{-\leftstackrelbwd}%
	\leavevmode\ifthenelse{\lengthtest{\leftstackrelawd>0pt}}%
	{\kern-.5\leftstackrelawd}{}\mathrel{\mathop{#2}\limits^{#1}}}

\theoremstyle{plain}
\newtheorem{thm}{Theorem}[section]

\newtheorem{cor}[thm]{Corollary}
\newtheorem{prop}[thm]{Proposition}
\newtheorem*{thm*}{Theorem}

\theoremstyle{definition}

\newtheorem{?}[thm]{Problem}

\newtheorem{innercustomthm}{Theorem}
\newenvironment{customthm}[1]
{\renewcommand\theinnercustomthm{#1}\innercustomthm}
{\endinnercustomthm}

\newtheorem{innercustomrmk}{Remark}
\newenvironment{customrmk}[1]
{\renewcommand\theinnercustomrmk{#1}\innercustomrmk}
{\endinnercustomrmk}

\newcommand{\KN}{\mathbin{\bigcirc\mspace{-15mu}\wedge\mspace{3mu}}}
\newcommand{\ep}{\varepsilon}
\renewcommand{\epsilon}{\varepsilon}
\makeatletter
\def\@cite#1#2{[\textbf{#1\if@tempswa , #2\fi}]}
\def\@biblabel#1{[\textbf{#1}]}
\makeatother

\def\XXint#1#2#3{{\setbox0=\hbox{$#1{#2#3}{\int}$}
		\vcenter{\hbox{$#2#3$}}\kern-.5\wd0}}

\makeatletter
\newcommand*{\defeq}{\mathrel{\rlap{%
			\raisebox{0.3ex}{$\m@th\cdot$}}%
		\raisebox{-0.3ex}{$\m@th\cdot$}}%
	=}
\makeatother

\makeatletter
\newcommand*{\eqdef}{=\mathrel{\rlap{%
			\raisebox{0.3ex}{$\m@th\cdot$}}%
		\raisebox{-0.3ex}{$\m@th\cdot$}}%
	}
\makeatother

\newcounter{marnote}

\makeatletter
\def\underbracex#1#2{\mathop{\vtop{\m@th\ialign{##\crcr
				$\hfil\displaystyle{#2}\hfil$\crcr
				\noalign{\kern3\p@\nointerlineskip}%
				#1\crcr\noalign{\kern3\p@}}}}\limits}

\def\upbracefilla{$\m@th \setbox\z@\hbox{$\braceld$}%
	\bracelu\leaders\vrule \@height\ht\z@ \@depth\z@\hfill 
	\kern\p@\vrule \@width\p@\kern\p@\vrule \@width\p@\kern\p@\vrule \@width\p@
	$}

\def\upbracefillb{$\m@th \setbox\z@\hbox{$\braceld$}%
	\vrule \@width\p@\kern\p@\vrule \@width\p@\kern\p@\vrule \@width\p@\kern\p@
	\leaders\vrule \@height\ht\z@ \@depth\z@\hfill\bracerd
	\braceld\leaders\vrule \@height\ht\z@ \@depth\z@\hfill
	\kern\p@\vrule \@width\p@\kern\p@\vrule \@width\p@\kern\p@\vrule \@width\p@
	$}

\def\upbracefillc{$\m@th \setbox\z@\hbox{$\braceld$}%
	\vrule \@width\p@\kern\p@\vrule \@width\p@\kern\p@\vrule \@width\p@\kern\p@
	\leaders\vrule \@height\ht\z@ \@depth\z@\hfill
	\kern\p@\vrule \@width\p@\kern\p@\vrule \@width\p@\kern\p@\vrule \@width\p@
	$}

\def\upbracefilld{$\m@th \setbox\z@\hbox{$\braceld$}%
	\vrule \@width\p@\kern\p@\vrule \@width\p@\kern\p@\vrule \@width\p@\kern\p@
	\leaders\vrule \@height\ht\z@ \@depth\z@\hfill\braceru$}

\def\upbracefillbd{$\m@th \setbox\z@\hbox{$\braceld$}%
	\vrule \@width\p@\kern\p@\vrule \@width\p@\kern\p@\vrule \@width\p@\kern\p@
	\bracerd\braceld
	\leaders\vrule \@height\ht\z@ \@depth\z@\hfill\braceru$}

\makeatother

\begin{document}
	\title{\vspace*{-10mm}Fully nonlinear Yamabe-type problems on non-compact manifolds}
	\author{Jonah A. J. Duncan\footnote{Johns Hopkins University, 404 Krieger Hall, Department of Mathematics, 3400 N.~Charles St, Baltimore, MD 21218, US. Email: jdunca33@jhu.edu} ~and Yi Wang\footnote{Johns Hopkins University, 404 Krieger Hall, Department of Mathematics, 3400 N.~Charles St, Baltimore, MD 21218, US. Email: ywang261@jhu.edu; research partially supported by NSF DMS-184503}}	
	\maketitle
	\vspace*{-4mm}
	\begin{abstract}
\vspace*{2mm}We obtain existence results for a class of fully nonlinear Yamabe-type problems on non-compact manifolds, addressing both the so-called positive and negative cases. We also give explicit examples of manifolds with warped product ends and asymptotically flat ends satisfying the hypotheses of our theorems.
	\end{abstract}

\section{Introduction}

Let $(M^n,g_0)$ be a Riemannian manifold of dimension $n \geq 3$, and denote by
\begin{align*}
A_{g_0} = \frac{1}{n-2}\bigg(\operatorname{Ric}_{g_0} - \frac{R_{g_0}}{2(n-1)}g_0\bigg)
\end{align*}
the Schouten tensor of $g_0$, where $\operatorname{Ric}_{g_0}$ and $R_{g_0}$ are the Ricci curvature tensor and scalar curvature of $g_0$, respectively. By the Ricci decomposition $\operatorname{Riem}_{g_0} = W_{g_0} + A_{g_0} \KN g_0$ of the Riemann curvature tensor and the conformal invariance of the $(1,3)$-Weyl tensor $g_0^{-1}W_{g_0}$, the Schouten tensor completely determines the conformal transformation properties of the Riemann curvature tensor and is thus a central quantity in conformal geometry. 

Since the work of Viaclovsky \cite{Via00a} and Chang, Gursky \& Yang \cite{CGY02a}, there has been significant interest in fully nonlinear equations and differential inclusions involving the eigenvalues of the Schouten tensor. To describe these problems, let $\lambda(g^{-1}A_g)$ denote the eigenvalues of the $(1,1)$-tensor $g^{-1}A_g$, and suppose that
\begin{align}
& \Gamma\subset\mathbb{R}^n\text{ is an open, convex, connected symmetric cone with vertex at 0}, \label{21'} \\
& \Gamma_n^+ = \{\lambda\in\mathbb{R}^n: \lambda_i > 0 ~\forall ~1\leq i \leq n\} \subseteq \Gamma \subseteq \Gamma_1^+ =  \{\lambda\in\mathbb{R}^n : \lambda_1+\dots+\lambda_n > 0\}, \label{22'} \\
& f\in C^\infty(\Gamma)\cap C^0(\overline{\Gamma}) \text{ is concave, 1-homogeneous and symmetric in the }\lambda_i, \label{23'}  \\
& f>0 \text{ in }\Gamma, \quad f = 0 \text{ on }\partial\Gamma, \quad f_{\lambda_i} >0 \text{ in } \Gamma \text{ for }1 \leq i \leq n. \label{24'}
\end{align}
We first identify two problems in the so-called positive case: to determine when there exists a conformal metric $g\in[g_0]$ satisfying 
\begin{align}\label{-1}
\lambda(g^{-1}A_g)\in\Gamma \quad \text{on }M,
\end{align}
and to determine when there exists $g\in[g_0]$ satisfying 
\begin{align}\label{-1'}
\lambda(g^{-1}A_g)\in\partial\Gamma \quad \text{on }M.
\end{align}
Likewise, in the so-called negative case one is interested in the existence of a conformal metric $g\in[g_0]$ satisfying
\begin{align}\label{-2}
\lambda(-g^{-1}A_g)\in\Gamma \quad \text{on }M,
\end{align}
and the existence of $g\in[g_0]$ satisfying 
\begin{align}\label{-2'}
\lambda(-g^{-1}A_g)\in\partial\Gamma \quad \text{on }M.
\end{align}
If there exists a metric $g$ satisfying \eqref{-1} (resp.~\eqref{-2}), one can then consider the associated Yamabe-type problem of prescribing $f(\lambda(g^{-1}A_g))$ (resp.~$f(\lambda(-g^{-1}A_g))$) to be a positive constant or function. Note that either of \eqref{-1'} or \eqref{-2'} already implies that $g$ solves the Yamabe-type equation $f(\lambda(g^{-1}A_g))=0$ on $M$. 

Important examples of $(f,\Gamma)$ satisfying \eqref{21'}--\eqref{24'} are given by $(f,\Gamma) = (\sigma_k^{1/k}, \Gamma_k^+)$ for $1\leq k \leq n$, where $\sigma_k:\mathbb{R}^n\rightarrow \mathbb{R}$ is the $k$'th elementary symmetric polynomial defined by
\begin{align*}
\sigma_k(\lambda_1,\dots,\lambda_n) = \sum_{1\leq i_1<\dots<i_k\leq n}\lambda_{i_1}\dots\lambda_{i_k},
\end{align*}
and $\Gamma_k^+ = \{\lambda\in\mathbb{R}^n:\sigma_j(\lambda)>0  \text{ for all }j=1,\dots,k\}$. When $k=1$, \eqref{-1'} and \eqref{-2'} are equivalent to $R_g = 0$, and \eqref{-1} (resp.~\eqref{-2}) is equivalent to $R_g>0$ (resp.~$R_g<0$). When $M$ is closed, these three possibilities are mutually exclusive within a fixed conformal class, and finding $g\in[g_0]$ such that $\sigma_1(\lambda(g^{-1}A_g))$ is constant is the classical Yamabe problem, solved by Yamabe, Trudinger, Aubin and Schoen \cite{Yam60, Tru68, Aub76, Sch84}. When $M$ is closed and $\Gamma\not=\Gamma_1^+$, the existence of conformal metrics satisfying differential inclusions of the form \eqref{-1} and/or \eqref{-1'} has been considered in \cite{CGY02a, CD10, GLW10, DN22, LN20}, for example. The problem of then deforming to a conformal metric with $f(\lambda(g^{-1}A_g))$ constant has been considered in \cite{GW03a, LL03, LL05, GW06, GV07, LN14, CGY02b, STW07}, for example. Of particular interest is the case $(f,\Gamma) = (\sigma_k^{1/k},\Gamma_k^+)$ for $k\geq 2$, often referred to as the $\sigma_k$-Yamabe problem in the positive case. For similar Yamabe-type problems on closed manifolds in the negative case (assuming the background metric satisfies \eqref{-2}), see e.g.~\cite{GV03b, LN20b}. By the transformation law for the Schouten tensor, 
\begin{align*}
A_{g_u} = -\nabla_{g_0}^2 u + du\otimes du - \frac{1}{2}|\nabla_{g_0} u|_{g_0}^2 g_0 + A_{g_0}, \quad g_u = e^{2u}g_0,
\end{align*}
these Yamabe-type problems are equivalent to solving a fully nonlinear, non-uniformly elliptic equation in $u$ when $\Gamma \not = \Gamma_1^+$. Moreover, the equations \eqref{-1'} and \eqref{-2'} are locally strictly elliptic when $(1,0,\dots,0)\in\Gamma$ (equivalently, when $\overline{\Gamma_n^+}\backslash\{0\}\subset \Gamma$), and degenerate elliptic when $\Gamma \not=\Gamma_1^+$ and $(1,0,\dots,0)\in\partial\Gamma$. We refer the reader to the discussion following Remark \ref{68} for further details on these conditions.

In this paper, we study the differential inclusions \eqref{-1}--\eqref{-2'} and their associated Yamabe-type problems on non-compact manifolds. Before stating our main results, we recall from \cite{LN14b} the quantity $\mu_\Gamma^+$, defined to be the number uniquely determined by $\Gamma$ such that
\begin{align*}
(-\mu_\Gamma^+, 1, \dots, 1)\in\partial \Gamma.  
\end{align*}
It is easy to see that $\mu_\Gamma^+\in[0,n-1]$ and $\mu_{\Gamma_k^+}^+ = \frac{n-k}{k}$, and thus $\mu_{\Gamma_k^+}^+>1$ if and only if $k<\frac{n}{2}$.

Our first main result concerns the negative cases \eqref{-2} and \eqref{-2'}:

\begin{thm}\label{H}
	Suppose $(f,\Gamma)$ satisfies \eqref{21'}--\eqref{24'} and $\mu_\Gamma^+>1$. Let $(M^n,g_0)$ be a smooth complete non-compact manifold such that, for some constant $c>0$ and some compact set $K_0\subset M$, it holds that $f(\lambda(-g_0^{-1}A_{g_0})) \geq c>0$ and $\lambda(-g_0^{-1}A_{g_0})\in\Gamma$ on $M\backslash K_0$.
	\begin{enumerate}[1.]
		\item If $(1,0,\dots,0)\in\Gamma$, then there exists a complete $C^{1,1}_{\operatorname{loc}}$ metric $g\in[g_0]$ satisfying 
		\begin{align}\label{065}
		\lambda(-g^{-1}A_g)\in\overline{\Gamma}
		\end{align}
		a.e.~on $M$. More precisely, at least one of the following holds: 
		\begin{enumerate}
			\item For any positive $\psi\in C^\infty_{\operatorname{loc}}(M)$ with $\|\ln\psi\|_{L^\infty(M)}<\infty$, there exists a complete $C^\infty_{\operatorname{loc}}$ metric $g\in[g_0]$ satisfying
			\begin{align}\label{054}
			f(\lambda(-g^{-1}A_g)) = \psi, \quad \lambda(-g^{-1}A_g)\in\Gamma \quad \text{on }M.
			\end{align}
			\item There exists a complete $C^{1,1}_{\operatorname{loc}}$ metric $g\in[g_0]$ satisfying \eqref{-2'} a.e.
		\end{enumerate}
		
		\item If $(1,0,\dots,0)\in\partial\Gamma$, then there exists a complete $C^{0,1}_{\operatorname{loc}}$ metric $g\in[g_0]$ satisfying \eqref{065} in the viscosity sense on $M$. More precisely, at least one of the following holds: 
		\begin{enumerate}
			\item For any positive $\psi\in C^\infty_{\operatorname{loc}}(M)$ with $\|\ln\psi\|_{L^\infty(M)}<\infty$, there exists a complete $C^{0,1}_{\operatorname{loc}}$ metric $g\in[g_0]$ satisfying \eqref{054} in the viscosity sense.
			\item There exists a complete $C^{0,1}_{\operatorname{loc}}$ metric $g\in[g_0]$ satisfying \eqref{-2'} in the viscosity sense.
		\end{enumerate}
	\end{enumerate}
\end{thm}

\noindent We refer the reader to Section \ref{ex2} for examples of manifolds $(M,g_0)$ with warped product ends satisfying the assumptions of Theorem \ref{H}.

We recall that for a continuous function $u$ on $M$, $g_u = e^{2u}g_0$ is a viscosity subsolution (resp.~viscosity supersolution) to \eqref{054} if for any $x_0\in M$ and $\phi \in C^2(M)$ satisfying $u(x_0) = \phi(x_0)$ and $u(x) \leq \phi(x)$ near $x_0$ (resp.~$u(x) \geq \phi(x)$ near $x_0$), it holds that $\lambda(-g_{\phi}^{-1}A_{g_\phi})\in\{\lambda\in\Gamma: f(\lambda)\geq \psi\}$ (resp.~$\lambda(-g_{\phi}^{-1}A_{g_\phi})\in\mathbb{R}^n\backslash\{\lambda\in\Gamma: f(\lambda)> \psi\}$). We call $g_u$ a viscosity solution if it is both a viscosity subsolution and supersolution. These notions for fully nonlinear Yamabe-type equations were first considered by Li in \cite{Li09}. 

\begin{customrmk}{1}\label{68}
	A key feature of Theorem \ref{H} is that we do not assume $g_0$ satisfies $\lambda(-g_0^{-1}A_{g_0})\in\Gamma$ on all of $M$; this is in contrast to existing literature on fully nonlinear Yamabe-type problems in the negative case on closed manifolds (see e.g.~\cite{GV03b, LN20b}) and non-compact manifolds (see e.g.~\cite{Yuan22, FSY20}). Our proof of Theorem \ref{H} uses an exhaustion argument, which relies on the existence of solutions to corresponding Dirichlet boundary value problems, recently obtained by the first author and Nguyen in \cite{DN23} (see also \cite{GSW11} for related results with the Ricci tensor in place of the Schouten tensor). In \cite{DN23}, there is no assumption on the background metric other than smoothness. On the other hand, the assumption $f(\lambda(-g_0^{-1}A_{g_0})) \geq c>0$ on $M\backslash K_0$ in Theorem \ref{H} is only used to ensure completeness of the limiting metric obtained from the exhaustion argument. We refer the reader to Section \ref{inc} for some further existence results for possibly incomplete metrics under weaker/different assumptions to Theorem \ref{H}. 
\end{customrmk}

\begin{customrmk}{2}\label{83}
	 It would be interesting to determine whether one or both assertions (a) and (b) hold in Theorem \ref{H}, perhaps under more stringent conditions on $g_0$ at infinity. Even in the case of scalar curvature, the assumption that $R_{g_0} \leq -\ep<0$ outside a compact set is insufficient to guarantee the existence of a conformal metric with constant negative scalar curvature, see e.g.~\cite[Example 6.1]{AM88-2}. Following the classical work \cite{ACF92}, in recent works such as \cite{AILA18, HN23}, existence results for constant negative scalar curvature metrics are obtained by imposing weak notions of asymptotic (local) hyperbolicity on the background metric -- see also the references therein for related results.
\end{customrmk}

\begin{customrmk}{3}
	In \cite{DN22}, the first author and Nguyen proved (among other results) the following dichotomy: if $M$ is closed and $\lambda(g_0^{-1}A_{g_0})\in\overline{\Gamma}$ on $M$, then either there exists a smooth metric $g\in[g_0]$ satisfying $\lambda(g^{-1}A_g)\in\Gamma$ on $M$, or there exists a $C^{1,1}$ metric $g\in[g_0]$ satisfying $\lambda(g^{-1}A_g)\in\partial\Gamma$ a.e.~on $M$. These two possibilities are mutually exclusive by the maximum principle, which in general does not apply in the non-compact setting (cf.~Remark \ref{83}).
\end{customrmk}

Let us now explain the condition $(1,0,\dots,0)\in\Gamma$ in Case 1 of Theorem \ref{H} in more detail. It is shown in \cite[Appendix A]{DN22} that if $\Gamma$ satisfies \eqref{21'} and \eqref{22'}, then $(1,0,\dots,0)\in\Gamma$ if and only if $\overline{\Gamma_n^+}\backslash\{0\}\subset\Gamma$, which occurs if and only $\Gamma = (\widetilde{\Gamma})^\tau$ for some $\tau\in(0,1)$ and $\widetilde{\Gamma}$ satisfying \eqref{21'} and \eqref{22'}, where $(\widetilde{\Gamma})^\tau \defeq \{\lambda\in\mathbb{R}^n: \tau\lambda + (1-\tau)\sigma_1(\lambda)e\in\widetilde{\Gamma}\}$ and $e=(1,\dots,1)$. Defining $f^\tau(\lambda) = f(\tau\lambda + (1-\tau)\sigma_1(\lambda)e)$, an equivalent formulation of Case 1 in Theorem \ref{H} is therefore as follows:
	
	\begin{cor}\label{F'}
		Suppose $(f,\Gamma)$ satisfies \eqref{21'}--\eqref{24'}, let $\tau<1$ and suppose $\mu_{\Gamma^\tau}^+>1$. Let $(M^n,g_0)$ be a smooth complete non-compact manifold such that, for some constant $c>0$ and some compact set $K_0\subset M$, it holds that $f^\tau(\lambda(-g_0^{-1}A_{g_0})) \geq c>0$ and $\lambda(-g_0^{-1}A_{g_0})\in\Gamma^\tau$ on $M\backslash K_0$. Then there exists a complete $C^{1,1}_{\operatorname{loc}}$ metric $g\in[g_0]$ satisfying 
			\begin{align*}
			\lambda(-g^{-1}A_g)\in\overline{\Gamma^\tau}
			\end{align*}
			a.e.~on $M$. More precisely, at least one of the following holds: 
			\begin{enumerate}
				\item For any positive $\psi\in C^\infty_{\operatorname{loc}}(M)$ with $\|\ln\psi\|_{L^\infty(M)}<\infty$, there exists a complete $C^\infty_{\operatorname{loc}}$ metric $g\in[g_0]$ satisfying
				\begin{align}\label{54'}
				f^\tau(\lambda(-g^{-1}A_g)) = \psi, \quad \lambda(-g^{-1}A_g)\in\Gamma^\tau \quad \text{on }M.
				\end{align}
				\item There exists a complete $C^{1,1}_{\operatorname{loc}}$ metric $g\in[g_0]$ satisfying 
				\begin{align*}
				\lambda(-g^{-1}A_g)\in\partial\Gamma^\tau \quad \text{a.e. on }M.
				\end{align*}
			\end{enumerate}
	\end{cor}

\noindent We next observe that if $t\in(-\infty,1)$ and 
	\begin{align*}
	A_g^t \defeq \frac{1}{n-2}\bigg(\operatorname{Ric}_g - \frac{t}{2(n-1)}R_g g\bigg)
	\end{align*}
	denotes the so-called trace-modified Schouten tensor, then $A_g^t = \tau^{-1}[\tau A_g + (1-\tau)\sigma_1(g^{-1}A_g)g]$ for $\tau = (1+\frac{1-t}{n-2})^{-1}\in(0,1)$. Therefore, Corollary \ref{F'} encompasses equations involving the trace-modified Schouten tensor. A particular example of interest is when $t=0$, or equivalently $\tau = \tau_0 \defeq \frac{n-2}{n-1}$, in which case $A_g^0$ is a positive multiple of the Ricci tensor. It is routine to verify that $\mu_{\Gamma^{\tau_0}}^+ > 1$ if and only if $\Gamma\not=\Gamma_n^+$, and thus we immediately obtain from Corollary \ref{F'} the following: 
	
	\begin{cor}\label{F''}
		Suppose that $(f,\Gamma)$ satisfies \eqref{21'}--\eqref{24'} and $\Gamma\not=\Gamma_n^+$. Let $(M^n,g_0)$ be a smooth complete non-compact manifold such that, for some constant $c>0$ and some compact set $K_0\subset M$, it holds that $f(\lambda(-g_0^{-1}\operatorname{Ric}_{g_0})) \geq c>0$ and $\lambda(-g_0^{-1}\operatorname{Ric}_{g_0})\in\Gamma$ on $M\backslash K_0$. Then there exists a complete $C^{1,1}_{\operatorname{loc}}$ metric $g\in[g_0]$ satisfying 
		\begin{align*}
		\lambda(-g^{-1}\operatorname{Ric}_g)\in\overline{\Gamma}
		\end{align*}
		a.e.~on $M$. More precisely, at least one of the following holds: 
		\begin{enumerate}
			\item For any positive $\psi\in C^\infty_{\operatorname{loc}}(M)$ with $\|\ln\psi\|_{L^\infty(M)}<\infty$, there exists a complete $C^\infty_{\operatorname{loc}}$ metric $g\in[g_0]$ satisfying
			\begin{align*}
			f(\lambda(-g^{-1}\operatorname{Ric}_g)) = \psi, \quad \lambda(-g^{-1}\operatorname{Ric}_g)\in\Gamma \quad \text{on }M.
			\end{align*}
			\item There exists a complete $C^{1,1}_{\operatorname{loc}}$ metric $g\in[g_0]$ satisfying 
			\begin{align*}
			\lambda(-g^{-1}\operatorname{Ric}_g)\in\partial\Gamma \quad \text{a.e. on }M.
			\end{align*}
		\end{enumerate}
	\end{cor}

\begin{customrmk}{4}\label{69}
The fact that we cannot assert better than Lipschitz regularity in Case 2 of Theorem \ref{H} is related to the lack of a known \textit{a priori} second derivative estimate on solutions to \eqref{054} when $(1,0,\dots,0)\in\partial\Gamma$. In general, this Lipschitz regularity cannot be improved to $C^1$ regularity -- see the counterexamples of Li \& Nguyen \cite{LN20b} and Li, Nguyen \& Xiong \cite{LNX22}. On the other hand, the trace-modified equation \eqref{54'} admits \textit{a priori} second derivative estimates when $\tau<1$ which depend on the $C^1$ norm of the conformal factor -- see Gursky \& Viaclovsky \cite{GV03b} for the case $(f,\Gamma) = (\sigma_k^{1/k},\Gamma_k^+)$, and Guan \cite{Guan08} for the general case. In upcoming work of Chu, Li \& Li \cite{CLL23-2}, local interior second derivative estimates depending only on gradient bounds and one-sided $C^0$ bounds are obtained for $\tau<1$. These local estimates will be used in our proof of Theorem \ref{H}. 
\end{customrmk}

We now turn to the positive cases \eqref{-1} and \eqref{-1'}. Our main result in this direction is as follows:

\begin{thm}\label{A'}
	Suppose $(f,\Gamma)$ satisfies \eqref{21'}--\eqref{24'}, and let $(M^n,g_0)$ be a smooth complete non-compact manifold such that $\lambda(g_0^{-1}A_{g_0})\in\Gamma$ on $M$. Suppose further that there exists a smooth complete conformal metric $g_v = e^{2v}g_0$ such that $\sup_M v<\infty$ and $R_{g_v} \leq 0$ on $M$. Then
	\begin{enumerate}
		\item If $\inf_M f(\lambda(g_0^{-1}A_{g_0})) = 0$, there exists a complete $C^{1,1}_{\operatorname{loc}}$ metric $g\in[g_0]$ satisfying \eqref{-1'} a.e.
		
		\item If $\inf_M f(\lambda(g_0^{-1}A_{g_0})) >0$, then for any positive $\psi\in C^\infty_{\operatorname{loc}}(M)$ with $\|\ln\psi\|_{L^\infty(M)}<\infty$, there exists a complete $C^\infty_{\operatorname{loc}}$ metric $g\in[g_0]$ satisfying 
		\begin{align}\label{81}
		f(\lambda(g^{-1}A_g)) = \psi, \quad \lambda(g^{-1}A_g)\in\Gamma \quad \text{on }M. 
		\end{align}
		\end{enumerate}
\end{thm}

\begin{customrmk}{5}
	Note that in contrast to Theorem \ref{H}, we do make a global admissibility assumption on the background metric in Theorem \ref{A'}, namely that $\lambda(g_0^{-1}A_{g_0})\in\Gamma$ on $M$ (cf.~Remark \ref{68}). The reason for this is as follows: as in the proof of Theorem \ref{H}, the proof of Theorem \ref{A'} uses an exhaustion argument which relies on the existence of solutions to suitable Dirichlet boundary value problems on compact manifolds with boundary; such solutions were obtained by Guan \cite{Gua07} under the assumption that the background metric satisfies $\lambda(g_0^{-1}A_{g_0})\in\Gamma$. It would be interesting to determine whether Case 1 of Theorem \ref{A'} holds assuming only $\lambda(g_0^{-1}A_{g_0})\in\overline{\Gamma}$ on $M$.
\end{customrmk}

One setting in which Theorem \ref{A'} applies is as follows. Suppose that $(M,g_0)$ is asymptotically flat (of suitable order and regularity to be specified below) and $\lambda(g_0^{-1}A_{g_0})\in\Gamma$ on $M$. In particular, $R_{g_0}>0$ and thus the work of Maxwell \cite[Proposition 3]{Max05} implies the existence of an asymptotically flat, scalar flat conformal metric $g_{v}$. Moreover, since concavity and homogeneity of $f$ imply 
\begin{align}\label{10}
f(\lambda) \leq f\bigg(\frac{\sigma_1(\lambda)}{n}e\bigg) + \nabla f\bigg(\frac{\sigma_1(\lambda)}{n}e\bigg)\cdot\bigg(\lambda - \frac{\sigma_1(\lambda)}{n}e\bigg) = \frac{f(e)}{n}\sigma_1(\lambda) \quad \text{if }\lambda\in\Gamma,
\end{align}
and since the scalar curvature of an asymptotically flat metric tends to zero at infinity, our assumptions on $g_0$ therefore imply $\inf_M f(\lambda(g_0^{-1}A_{g_0})) = 0$. Therefore, the hypotheses in Case 1 of Theorem \ref{A'} are satisfied. In fact, by a small refinement to the proof of Theorem \ref{A'}, we can obtain an asymptotically flat solution to \eqref{-1'} in this case, at least in the $C^0$ sense. We refer the reader to Section \ref{ex} for the relevant notation and definitions used in the following:

\begin{cor}\label{A}
	Suppose $\Gamma$ satisfies \eqref{21'} and \eqref{22'}, and let $(M^n,g_0)$ be a smooth complete $W^{2,p}_{-\tau}$-asymptotically flat manifold satisfying $\lambda(g_0^{-1}A_{g_0})\in\Gamma$ on $M^n$, where $\tau\in(0,n-2)$ and $p>\frac{n}{2}$. Then there exists a $C^{1,1}_{\operatorname{loc}}$ metric $g\in[g_0]$ which is $C^0_{-\tau}$-asymptotically flat and satisfies \eqref{-1'} a.e.
\end{cor}

\begin{customrmk}{6}\label{79}
	We only assume in Corollary \ref{A} that $p>\frac{n}{2}$. If $g_0$ is $W^{2,p}_{-\tau}$-asymptotically flat and one imposes the stronger assumption $p>n$, then the Morrey embedding theorem for weighted Sobolev spaces (see \cite[Theorem 1.2]{Bart86}) implies that $g_0$ is $C^1_{-\tau}$-asymptotically flat. Moreover, a recent result of Zhang \cite[Theorem 3.4]{Z18} shows that a $C^1_{-\tau}$-asymptotically flat manifold $(M,g_0)$ with $\operatorname{Ric}_{g_0} \geq 0$ is isometric to Euclidean space. Since $\lambda(g_0^{-1}A_{g_0})\in\Gamma_k^+$ implies $\operatorname{Ric}_{g_0}\geq 0$ if $k\geq \frac{n}{2}$ \cite{GVW03}, it follows that there is no $W^{2,p}_{-\tau}$-asymptotically flat manifold satisfying $\lambda(g_0^{-1}A_{g_0})\in\Gamma_k^+$ on $M^n$ for $p>n$ and $k\geq \frac{n}{2}$ (see also the related non-existence results of Li \& Nguyen \cite{LN20}). On the other hand, we provide in Section \ref{ex} examples of manifolds satisfying the hypotheses of Corollary \ref{A} whenever $\mu_\Gamma^+>1$ (which we recall is equivalent to $k<\frac{n}{2}$ when $\Gamma = \Gamma_k^+$). Our examples consist of Riemannian Schwarzchild-type metrics and suitable perturbations. 
\end{customrmk}

Asymptotically flat metrics have also been considered in the context of fully nonlinear Yamabe-type problems by other authors. Suppose that $\Gamma$ satisfies \eqref{21'}, \eqref{22'} and $\mu_\Gamma^+>1$, and let $(M,g_0)$ be a closed manifold satisfying $\lambda(g_0^{-1}A_{g_0})\in\Gamma$ on $M^n$. In \cite{LN20}, Li \& Nguyen prove (among other results) the existence of a Green's function for $\Gamma$, for any given set of poles $\{p_1,\dots,p_m\}$ and corresponding weights. In particular, the metric $g$ they construct belongs to $C^{1,1}_{\operatorname{loc}}(M\backslash\{
p_1,\dots,p_m\})$, satisfies $\lambda(g^{-1}A_g)\in\partial\Gamma$ a.e.~in $M\backslash\{
p_1,\dots,p_m\}$, and has an asymptotically flat end at each pole. We point out that the assumptions in Corollary \ref{A} do not necessarily imply that $(M,g_0)$ admits a smooth conformal compactification $(\overline{M},\overline{g}_0)$ satisfying $\lambda(\overline{g}_0^{-1}A_{\overline{g}_0})\in\Gamma$ on $\overline{M}$, and thus Corollary \ref{A} does not immediately follow from the results in \cite{LN20}. As in \cite{LN20}, part of our motivation for studying the equation $\lambda(g^{-1}A_g)\in\partial\Gamma$ stems from the expectation that, in the study of compactness problems for the $\sigma_k$-Yamabe equation on closed manifolds, such metrics may arise as limits of appropriately rescaled sequences of blow-up solutions. We note that the existence of solutions and compactness of the solution set to the $\sigma_k$-Yamabe problem on closed, non-locally conformally flat manifolds when $2<k<\frac{n}{2}$ remains a major open problem in conformal geometry, as does compactness of the solution set when $k=2$.  

In related work, Ge, Wang \& Wu introduced in \cite{GWW14} ADM-type masses $m_k$ for asymptotically flat manifolds using the so-called Gauss-Bonnet curvatures $L_k$ when $1\leq k < \frac{n}{2}$ (see also the work of Li \& Nguyen \cite{LN13} for a related notion). For example, the second Gauss-Bonnet curvature is given by 
\begin{align*}
L_2 = \|W_g\|^2 + 8(n-2)(n-3)\sigma_2(\lambda(g^{-1}A_g)), 
\end{align*}
and thus $L_2$ is a positive multiple of $\sigma_2(\lambda(g^{-1}A_g))$ on locally conformally flat manifolds. Ge, Wang \& Wu prove a corresponding positive mass theorem in the graphical case in \cite{GWW14}, and in the conformally flat case in \cite{GWW14b}, assuming integrability and nonnegativity of $L_k$. See also \cite[Theorem 1.2]{LN14b} for a related result in the locally conformally flat case. It would be interesting to see if Corollary \ref{A} can be applied in the study of these generalised masses. \medskip

The plan of paper is as follows. In Section \ref{c2} we prove Theorem \ref{H} and give some further existence results for possibly incomplete metrics. In Section \ref{ex2} we give examples of manifolds with warped product ends satisfying the assumptions of Theorem \ref{H}. In Section \ref{pos} we prove Theorem \ref{A'} and Corollary \ref{A}. Finally, in Section \ref{ex} we give examples of asymptotically flat manifolds satisfying the hypotheses of Corollary \ref{A}. \medskip 

\noindent\textbf{Acknowledgements:} The authors would like to thank Matthew Gursky, Yanyan Li, Luc Nguyen and Yannick Sire for helpful discussions and comments.

\section{Existence results in the negative case}\label{c2}

\subsection{Proof of Theorem \ref{H}}

The proof of Theorem \ref{H} is via an exhaustion argument. One result that will be used to carry out this procedure is the following recent existence result of Duncan \& Nguyen \cite{DN23}: 

\begin{customthm}{A}[\cite{DN23}]\label{AA}
	\textit{Let $(N,g_0)$ be a smooth compact Riemannian manifold of dimension $n\geq 3$ with non-empty boundary $\partial N$, and suppose that $(f,\Gamma)$ satisfies \eqref{21'}--\eqref{24'} and $\mu_\Gamma^+>1$. Let $\psi\in C^\infty(N)$ be positive and $\xi\in C^\infty(\partial N)$. Then there exists a Lipschitz viscosity solution $g_u = e^{2u}g_0$ to 
	\begin{align}\label{61}
	\begin{cases}
	f(\lambda(-g_u^{-1}A_{g_u})) = \psi, \quad \lambda(-g_u^{-1}A_{g_u})\in\Gamma & \text{on }N\backslash\partial N \\
	u = \xi & \text{on }\partial N.  
	\end{cases}
	\end{align}
	Moreover, when $(1,0,\dots,0)\in\Gamma$, the solution $u$ is smooth and is the unique continuous viscosity solution to \eqref{61}. }
\end{customthm}

In fact, in \cite{DN23} the authors prove existence of solutions to a class of fully nonlinear Loewner-Nirenberg-type problems, although it will suffice to consider \eqref{61} with $\xi\equiv 0$ in this paper. Another important ingredient in the proof of Theorem \ref{H} is a local gradient and Hessian estimate on solutions to the equation in \eqref{61} depending only on an upper bound for the solution $u$: 

\begin{customthm}{B}[\cite{CLL23, CLL23-2}]\label{BB}
	\textit{Let $(N,g_0)$ be a smooth Riemannian manifold of dimension $n\geq 3$, possibly with non-empty boundary, and suppose that $(f,\Gamma)$ satisfies \eqref{21'}--\eqref{24'}. Fix $\tau\in[0,1]$ and suppose $\mu_{\Gamma^\tau}^+\not=1$. Let $\psi\in C^\infty(N)$ be positive and suppose $g_u = e^{2u}g_0$, $u\in C^4(B_r)$, satisfies 
		\begin{align*}
		f^\tau(\lambda(-g_u^{-1}A_{g_u})) = \psi, \quad  \lambda(-g_u^{-1}A_{g_u})\in\Gamma^\tau
		\end{align*}
		in a geodesic ball $B_r$ contained in the interior of $N$. Then 
		\begin{align*}
	|\nabla_{g_0} u|_{g_0}(x) \leq C(r^{-1} + e^{\sup_{B_r} u})\quad \text{for }x\in B_{r/2},
		\end{align*}
		where $C$ is a constant depending on $n,f,\Gamma, \|g_0\|_{C^3(B_r)}$ and  $\|\psi\|_{C^1(B_r)}$ but independent of $\tau$ and $\inf_{B_r} u$. If $\tau<1$ or $(1,0,\dots,0)\in\Gamma$, then 
	\begin{align*}
	|\nabla_{g_0}^2 u|_{g_0}(x) \leq C(r^{-2} + e^{2\sup_{B_r} u})\quad \text{for }x\in B_{r/2},
	\end{align*}
	where $C$ is a constant depending on $\tau, n,f,\Gamma, \|g_0\|_{C^4(B_r)}$ and  $\|\psi\|_{C^2(B_r)}$ but independent $\inf_{B_r} \psi$.}
\end{customthm}

We point out that the gradient estimate in Theorem \ref{BB} was obtained in the thesis of Khomrutai \cite{Kho09} in the special case that $(f,\Gamma) = (\sigma_k^{1/k},\Gamma_k^+)$ for either $k<n/2$ or $k\in\{n-1,n\}$. In \cite{DN23}, Duncan \& Nguyen extended the method of Khomrutai in the case $k<n/2$ to prove a local gradient estimate whenever $\mu_\Gamma^+>1$. For the general case when $\mu_\Gamma^+\not=1$, the gradient estimate in Theorem \ref{BB} was obtained by Chu, Li \& Li in \cite{CLL23} as an application of a very general Liouville-type theorem obtained therein, and counterexamples to such estimates in the case $\mu_\Gamma^+=1$ are also given. We have been informed by Chu, Li \& Li that the Hessian estimate stated in Theorem \ref{BB} will appear in \cite{CLL23-2}, under even more general assumptions on $f$. We point out that in this paper, we will only require Theorem \ref{BB} in the case $\mu_\Gamma^+>1$.

\begin{proof}[Proof of Theorem \ref{H}]
We first consider the case $(1,0,\dots,0)\in\Gamma$. Consider an exhaustion $\{M_j\}_{j=1}^\infty$ of $M^n$ by compact manifolds with boundary. By Theorem \ref{AA}, for each $j$ there exists a unique smooth solution $g_{u_j} = e^{2u_j}g_0$ to
\begin{align}\label{29}
\begin{cases}
f(-\lambda(g_{u_j}^{-1}A_{g_{u_j}})) = \psi, \quad \lambda(-g_{u_j}^{-1}A_{g_{u_j}}) \in\Gamma  & \text{in }M_j\backslash \partial M_j \\
u_j = 0  & \text{on }\partial M_j. 
\end{cases}
\end{align}

Now, by 1-homogeneity of $f$, we may assume without loss of generality that $f(\frac{1}{2},\dots,\frac{1}{2})=1$. With this normalisation, \eqref{10} then implies that the scalar curvature of the solution to \eqref{29} satisfies $R_{g_{u_j}} \leq -n(n-1)\inf_M \psi$. On the other hand, by the solution to the Loewner-Nirenberg problem on Riemannian manifolds due to Aviles \& McOwen \cite{AM88}, on any domain $\Omega\subset M^n$ there exists a smooth solution $g_{w_\Omega} = e^{2w_\Omega}g_0$ to 
\begin{align*}
\begin{cases}
R_{g_{w_\Omega}} = -n(n-1)\inf_M\psi & \text{on }\Omega \\
w_\Omega(x)\rightarrow+\infty & \text{as }\operatorname{dist}_{g_0}(x,\partial\Omega)\rightarrow 0. 
\end{cases}
\end{align*}
In particular, if $\Omega \Subset M_j$, then the comparison principle implies 
\begin{align*}
u_j \leq w_\Omega \quad \text{on }\Omega, 
\end{align*}
which gives a finite upper bound for $u_j$ on any compact subset of $\Omega$. It follows that for any compact subset $K\subset M^n$ and $j$ sufficiently large so that $K\subset M_j\backslash \partial M_j$, we have
\begin{align}\label{40}
\sup_K u_j \leq C_K,
\end{align} 
where here and henceforth, $C_K$ is a constant independent of $j$ but possibly depending on $K$ (we allow $C_K$ to change from line to line). By Theorem \ref{BB}, it follows from \eqref{40} that 
\begin{align}\label{41}
\|\nabla_{g_0} u_j\|_{C(K)} \leq C_K,
\end{align}
and
\begin{align}\label{46}
\|\nabla_{g_0}^2 u_j\|_{C(K)} \leq C_K. 
\end{align}

We now observe that at least one of the following statements holds: 
\begin{enumerate}
	\item There exists a subsequence of solutions to \eqref{29} (still denoted $u_j$) and a constant $C$ such that $\inf_{M_1} u_j \geq - C$ for all $j$. 
	\item There exists a subsequence of solutions to \eqref{29} (still denoted $u_j$) such that $\inf_{M_1} u_j\rightarrow-\infty$ as $j\rightarrow+\infty$. 
\end{enumerate}
Note that in both cases we are only considering the infimum over $M_1$. We consider these two cases separately. \medskip 

\noindent\textbf{\underline{Case 1}:} In this case, in combination with \eqref{40}--\eqref{46} we have
\begin{align*}
\|u_{j}\|_{C^2(M_1,g_0)} \leq C.
\end{align*}
$C^{2,\alpha}$ estimates in $M_1$ then follow from the regularity theory of Evans-Kyrlov \cite{Evans82, Kry82}, and higher order estimates follow from classical Schauder theory (see e.g.~\cite{GT}). Moreover, with a uniform $C^0$ estimate for $u_j$ established in $M_1$, \eqref{41} allows one to propagate this $C^0$ estimate to any compact subset. In combination with \eqref{46} and the aforementioned regularity theory, we therefore have for any $\ell\geq 2$ and any compact set $K\subset M$ such that $K\subset M_j\backslash\partial M_j$ the estimate
\begin{align}\label{50}
\|u_{j}\|_{C^\ell(K,g_0)} \leq C_{K,\ell}.
\end{align}

We now show that the solutions $u_j$ actually satisfy a \textit{global} lower bound independently of $j$; it is for this purpose we use the assumption $f(\lambda(-g_0^{-1}A_{g_0}))\geq c > 0$ and $\lambda(-g_0^{-1}A_{g_0})\in\Gamma$ on $M^n\backslash K_0$ for some compact set $K_0$. First observe that for fixed $j$, either $u_j$ attains its minimum in $K_0$ or in $M_j\backslash K_0$. In the former case, \eqref{50} implies $u_j \geq -C_{K_0}$ in $M_j$ for some constant $C_{K_0}$ independent of $j$. In the latter case, either $u_j$ attains its minimum at some $x_j\in\partial M_j$, in which case $u_j \geq 0$ on $M_j$, or it attains its minimum at some point $x_j$ in the interior of $M_j\backslash K_0$. At such a point, $\nabla_{g_0}^2 u_{j}(x_{j}) \geq 0$ and $du_j(x_j) =0$, and thus by the equation \eqref{29} we have
\begin{align*}
e^{2u_{j}(x_{j})}\sup_M\psi \geq f(-g_0^{-1}A_{g_0})(x_j) \geq c>0.
\end{align*}
This implies 
\begin{align*}
u_j(x_j) \geq \frac{1}{2}\ln ((\sup_M\psi)^{-1}c). 
\end{align*}
We have therefore shown 
\begin{align}\label{80'}
u_j \geq \min\bigg\{-C_{K_0}, 0, \frac{1}{2}\ln ((\sup_M\psi)^{-1}c)\bigg\} \quad \text{in }M_j. 
\end{align}

Using \eqref{50}, one can apply a standard diagonal argument to extract a subsequence which converges in $C^\infty_{\operatorname{loc}}(M^n)$ to some $u \in C^\infty_{\operatorname{loc}}(M^n)$ solving
\begin{align*}
f(\lambda(-g_u^{-1}A_{g_u})) = \psi, \quad \lambda(-g_u^{-1}A_{g_u})\in\Gamma  \quad \text{on }M^n. 
\end{align*}
Moreover, by \eqref{80'} $u$ is uniformly bounded from below on $M$. Since $g_0$ is complete, it is easy to see that this global lower bound on $u$ implies that the length of any divergent curve in $M$ with respect to $g_u$ is infinite, and therefore $g_u$ is also complete.  \medskip

\noindent\textbf{\underline{Case 2}:} Recall that in this case we assume $\inf_{M_1} u_j\rightarrow-\infty$ as $j\rightarrow\infty$. On each $M_j$ define
\begin{align*}
\widehat{u}_j = u_j- \inf_{M_1} u_j.
\end{align*}
Then 
\begin{align*}
f(\lambda(-g_{\widehat{u}_j}^{-1}A_{g_{\widehat{u}_j}})) = f(\lambda(-g_{\widehat{u}_j}^{-1}A_{g_{u_j}})) = e^{2\inf_{M_1} u_j} f(\lambda(-g_{u_j}^{-1}A_{g_{u_j}})) \stackrel{\eqref{29}}{=}  \psi e^{2\inf_{M_1}u_j}, 
\end{align*}
and so $\widehat{u}_j$ satisfies 
\begin{align}\label{30}
\begin{cases}
f(\lambda(-g_{\widehat{u}_j}^{-1}A_{g_{\widehat{u}_j}})) =  \psi e^{2\inf_{M_1} u_j}, \quad \lambda(-g_{\widehat{u}_j}^{-1}A_{g_{\widehat{u}_j}})\in\Gamma & \text{in }M_j\backslash \partial M_j \\
\widehat{u}_j = -\inf_{M_1} u_j & \text{on }\partial M_j. 
\end{cases}
\end{align}

Now, for each $j$ there exists $y_j\in M_1$ such that $\widehat{u}_j(y_j) = 0$. On the other hand, \eqref{41} and \eqref{46} immediately imply local first and second derivative estimates for $\widehat{u}_j$. Since $\widehat{u}_j(y_j) = 0$, the local first derivative estimate in particular implies a $C^0$ bound for $\widehat{u}_j$ on any compact subset of $M_j$. Therefore, for any compact subset $K\subset M^n$ and $j$ sufficiently large so that $K\subset M_j\backslash \partial M_j$, we obtain
\begin{align}\label{51}
\|\widehat{u}_j\|_{C^2(K,g_0)} \leq C_K.
\end{align} 
(Note that since $e^{2\inf_{M_1} u_j}\rightarrow 0$ as $j\rightarrow\infty$ and $\psi$ is uniformly bounded from above on $M$, the RHS of the equation in \eqref{30} tends to zero as $j\rightarrow \infty$. Thus we cannot necessarily apply the theory of Evans-Krylov \cite{Evans82, Kry82} to obtain higher order estimates.)

Following a similar argument to Case 1, we now show that the solutions $\widehat{u}_j$ satisfy a \textit{global} lower bound independently of $j$. For fixed $j$, if $\widehat{u}_j$ attains its minimum in $K_0$, then \eqref{51} implies $\widehat{u}_j \geq -C_{K_0}$ in $M_j$. If the minimum occurs at some $x_j\in \partial M_j$, then $\widehat{u}_j(x_j) = -\inf_{M_1} u_j \geq 0$ for sufficiently large $j$ (since we assume $\inf_{M_1}u_j\rightarrow-\infty$ as $j\rightarrow\infty$), and thus $\widehat{u}_j \geq 0$ in $M_j$ for large $j$. Finally, if the minimum occurs at some $x_j$ in the interior of $M_j\backslash K_0$, then $\nabla_{g_0}^2\widehat{u}_j(x_j) \geq 0$ and $d\widehat{u}_j (x_j) = 0$, and thus by the equation \eqref{29} we have
\begin{align*}
e^{2\widehat{u}_j(x_j) + 2\inf_{M_1} u_j} \sup\psi  \geq f(-g_0^{-1}A_{g_0})(x_j) \geq c >0. 
\end{align*}
This implies 
\begin{align*}
\widehat{u}_j(x_j) \geq \frac{1}{2}\ln ((\sup\psi)^{-1}c) - \inf_{M_1} u_j \geq \frac{1}{2}\ln ((\sup\psi)^{-1}c)
\end{align*}
for large $j$, and we have therefore shown 
\begin{align}\label{80''}
\widehat{u}_j \geq \min\bigg\{-C_{K_0}, 0, \frac{1}{2}\ln ((\sup_M\psi)^{-1}c)\bigg\} \quad \text{in }M_j \text{ for large }j.
\end{align}

Using \eqref{51}, one can apply a standard diagonal argument to extract a subsequence which converges in $C^{1,\alpha}_{\operatorname{loc}}(M^n)$ to some $\widehat{u}\in C^{1,1}_{\operatorname{loc}}(M^n)$ satisfying 
\begin{align*}
\lambda(-g_{\widehat{u}}^{-1}A_{g_{\widehat{u}}}) \in\partial\Gamma \quad \text{a.e.~in }M, 
\end{align*}
where we have used the fact that the RHS of the equation in \eqref{30} tends to zero as $j\rightarrow \infty$. By \eqref{80''}, $\widehat{u}$ is uniformly bounded from below on $M$, and thus $g_{\widehat{u}}$ is complete using the same reasoning as in Case 1. \medskip 

This completes the proof of Theorem \ref{H} in the case $(1,0,\dots,0)\in\Gamma$. The existence of a Lipschitz viscosity solution when $(1,0,\dots,0)\in\partial\Gamma$ is then obtained by an approximation argument using Corollary \ref{F'} and the fact that the $C^1$ estimates obtained above are independent of $\tau$; we refer the reader to proof of Theorem 1.3 in \cite{LN20b} for the details.
\end{proof}

\subsection{Further existence results for possibly incomplete metrics}\label{inc}

It is clear from the proof of Theorem \ref{H} that the only role played by the assumption $f(\lambda(-g_0^{-1}A_{g_0})) \geq c>0$ on $M\backslash K_0$ is to ensure completeness of the limiting metrics. We therefore have the following as an immediate consequence of the proof of Theorem \ref{H}: 

\begin{thm}\label{H'}
	Suppose $(f,\Gamma)$ satisfies \eqref{21'}--\eqref{24'} and $\mu_\Gamma^+>1$, and let $(M^n,g_0)$ be a smooth complete non-compact manifold. Then the same conclusion as in Theorem \ref{H} holds, except that the solutions may be incomplete at infinity. 
\end{thm} 

Moreover, if we assume further that the background metric has nonnegative scalar curvature, then we can obtain a solution to the degenerate equation:

\begin{thm}\label{H''}
	Assume in addition to the hypotheses in Theorem \ref{H'} that $R_{g_0}\geq 0$. If $(1,0,\dots,0)\in\Gamma$ (resp.~$(1,0,\dots,0)\in\partial\Gamma$), there exists a $C^{1,1}_{\operatorname{loc}}$ (resp.~$C^{0,1}_{\operatorname{loc}}$) metric $g\in[g_0]$, possibly incomplete at infinity, satisfying $\lambda(-g^{-1}A_g)\in\partial\Gamma$ a.e.~(resp.~in the viscosity sense) on $M$.
\end{thm}

\begin{proof}[Proof of Theorem \ref{H''}]
	Following the reasoning at the end of the proof of Theorem \ref{H}, it suffices to consider the case $(1,0,\dots,0)\in\Gamma$. Consider an exhaustion $\{M_j\}_{j=1}^\infty$ of $M$ by compact manifolds with boundary. Then by Theorem \ref{AA}, for each $j$ there exists a unique smooth $g_{u_j} = e^{2u_j}g_0$ to
	\begin{align*}
	\begin{cases}
	f(\lambda(-g_{u_j}^{-1}A_{g_{u_j}})) = j^{-1}, \quad \lambda(-g_{u_j}^{-1}A_{g_{u_j}})\in\Gamma & \text{on }M_j\backslash\partial M_j \\
	u_j = 0 & \text{on }\partial M_j.
	\end{cases}
	\end{align*}
	In particular, each metric $g_{u_j}$ has negative scalar curvature, and thus $w_j \defeq e^{\frac{n-2}{2}u_j}$ satisfies 
	\begin{align*}
	\begin{cases}
	-\Delta_{g_0} w_j + \frac{n-2}{4(n-1)}R_{g_0}w_j \leq 0 & \text{on }M_j\backslash \partial M_j \\
	w_j = 1 & \text{on }\partial M_j. 
	\end{cases}
	\end{align*}
	Since $R_{g_0} \geq 0$, the maximum principle implies $w_j \leq 1$ in $M_j$, or equivalently
	\begin{align}\label{71}
	u_j \leq 0 \quad \text{in }M_j. 
	\end{align}
	By Theorem \ref{BB}, it follows that for any compact $K\subset M^n$ and $j$ sufficiently large so that $K\subset M_j\backslash \partial M_j$, we have 
	\begin{align}\label{72}
	\|\nabla_{g_0} u_j \|_{C(K)} + \|\nabla_{g_0}^2 u_j \|_{C(K)} \leq C_K. 
	\end{align}
	We now split into two cases as in the proof of Theorem \ref{H}. First suppose that along a subsequence (still denoted $u_j$) it holds that $\inf_{M_1} u_j \geq -C$ for all $j$. Combining this with \eqref{71} and \eqref{72}, it follows that 
	\begin{align*}
	\|u_j\|_{C^2(K)} \leq C_K. 
	\end{align*}
	We can therefore extract a subsequence which converges in $C^{1,\alpha}_{\operatorname{loc}}(M^n)$ to some $u\in C^{1,1}_{\operatorname{loc}}(M^n)$ satisfying 
	\begin{align*}
	\lambda(-g_u^{-1}A_{g_u})\in\partial\Gamma \quad \text{a.e. on }M,
	\end{align*}
	as required. 
	
	Now suppose instead that along a subsequence (still denoted $u_j$) it holds that $\inf_{M_1} u_j \rightarrow -\infty$. Defining $\widehat{u}_j = u_j- \inf_{M_1} u_j$ as in Case 2 in the proof of Theorem \ref{H}, we see 
	\begin{align*}
	\begin{cases}
	f(\lambda(-g_{\widehat{u}_j}^{-1}A_{g_{\widehat{u}_j}})) =  j^{-1} e^{2\inf_{M_1} u_j}, \quad \lambda(-g_{\widehat{u}_j}^{-1}A_{g_{\widehat{u}_j}})\in\Gamma & \text{in }M_j\backslash \partial M_j \\
	\widehat{u}_j = -\inf_{M_1} u_j & \text{on }\partial M_j. 
	\end{cases}
	\end{align*}
	As justified in Case 2 in the proof of Theorem \ref{H}, it follows that
	\begin{align*}
	\| \widehat{u}_j\|_{C^2(K,g_0)} \leq C_K, 
	\end{align*}
	and since $j^{-1} e^{2\inf_{M_1} u_j}\rightarrow 0$ as $j\rightarrow\infty$, we can therefore extract a subsequence which converges in $C^{1,\alpha}_{\operatorname{loc}}(M^n)$ to some $\widehat{u}\in C^{1,1}_{\operatorname{loc}}(M^n)$ satisfying 
	\begin{align*}
	\lambda(-g_{\widehat{u}}^{-1}A_{g_{\widehat{u}}})\in\partial\Gamma \quad \text{a.e. on }M. 
	\end{align*} 
	This completes the proof of Theorem \ref{H''}. 
\end{proof}

\section{Examples in the negative case}\label{ex2}

In this section, we give a class of manifolds with warped product ends which satisfy the assumptions of Theorem \ref{H}. The set up for these examples is as follows. Suppose $M=M_0 \cup M_1\cup\dots M_m$ is a smooth $n$-manifold, where $M_0$ is a compact region and the $M_i$ ($i=1,\dots,m$) are mutually disjoint, non-compact ends such that each $M_i$ is disjoint from $M_0$ except along their common boundary. We further assume that each end splits as a product $M_i =[0,\infty)\times N_i$ for some compact $(n-1)$-manifold $N_i$, and we denote by $r$ the coordinate on the $[0,\infty)$ component. We also suppose that $N_i$ admits an Einstein metric $h_i$ which, after rescaling, we may assume has Einstein constant equal to $(n-2)k_i$ for some $k_i\in\{-1,0,1\}$. We then consider a metric $g_0$ on $M$ which takes the following form on each $M_i$ ($1\leq i \leq n$):
\begin{align*}
g_0|_{M_i} = dr^2 + \Phi_i^2(r) h_i. 
\end{align*}

\begin{prop}
	Suppose $(f,\Gamma)$ satisfies \eqref{21'}--\eqref{24'}, and suppose that for each $1\leq i \leq m$, $\Phi_i$ is a positive function on $[0,\infty)$ satisfying
	\begin{align}\label{82}
	\frac{\Phi'_i(r)^2 - k_i}{2\Phi_i(r)^2} \geq \ep>0  \quad \text{and} \quad \frac{2\Phi_i(r)\Phi_i''(r)}{\Phi_i'(r)^2 - k_i}  \geq -\mu_\Gamma^+ + 1 +\delta
	\end{align}
	for $r$ sufficiently large and for some positive constants $\ep$ and $\delta$. Then $g_0$ is a complete metric satisfying $f(\lambda(-g_0^{-1}A_{g_0})\geq c>0$ and $\lambda(-g_0^{-1}A_{g_0})\in\Gamma$ outside a compact subset of $M$, and thus satisfies the hypotheses of Theorem \ref{H}. 
\end{prop}

\begin{proof}
It suffices to consider just a single end $M_i$, and in what follows indexing is suppressed in order to ease notation (i.e.~we write $\Phi$ instead of $\Phi_i$, $k$ instead of $k_i$, and $h$ instead of $h_i$). A standard computation using e.g.~\cite[Proposition 9.106]{Bes87} gives the following formula for $\operatorname{Ric}_{g_0}$ on $M_i$:
\begin{align*}
\operatorname{Ric}_{g_0} = -\bigg(\frac{\Phi''(r)}{\Phi(r)}-(n-2)\frac{k-\Phi'(r)^2}{\Phi(r)^2}\bigg)g_0 - (n-2)\bigg(\frac{\Phi''(r)}{\Phi(r)} + \frac{k-\Phi'(r)^2}{\Phi(r)^2}\bigg)dr\otimes dr. 
\end{align*}
It follows that
\begin{align*}
R_{g_0} = -2(n-1)\frac{\Phi''(r)}{\Phi(r)} - (n-1)(n-2)\frac{\Phi'(r)^2}{\Phi(r)^2} + \frac{(n-1)(n-2)k}{\Phi(r)^2}
\end{align*}
and so 
\begin{align*}
A_{g_0} & = \frac{1}{n-2}\bigg(\operatorname{Ric}_{g_0} - \frac{R_{g_0}}{2(n-1)}g_0\bigg) = \frac{k-\Phi'(r)^2}{2\Phi(r)^2}g_0 - \bigg(\frac{\Phi''(r)}{\Phi(r)}+\frac{k-\Phi'(r)^2}{\Phi(r)^2}\bigg)dr\otimes dr.
\end{align*}
Therefore, the eigenvalues of $-g_0^{-1}A_{g_0}$ are given by $(\chi_1-\chi_2, \chi_1, \dots, \chi_1)$, where
\begin{align*}
\chi_1 = \frac{\Phi'(r)^2-k}{2\Phi(r)^2} \quad \text{and} \quad \chi_2 = -\frac{\Phi''(r)}{\Phi(r)} + \frac{\Phi'(r)^2-k}{\Phi(r)^2} = -\frac{\Phi''(r)}{\Phi(r)} + 2\chi_1. 
\end{align*}
Writing the vector of eigenvalues of $-g_0^{-1}A_{g_0}$ in the form $\chi_1(1-\frac{\chi_2}{\chi_1}, 1,\dots,1)$, we see that a sufficient condition for $f(-g_0^{-1}A_{g_0})\geq c>0$ and $\lambda(-g_0^{-1}A_{g_0})\in\Gamma$ for $r$ sufficiently large is that $\chi_1 \geq \ep >0$ and $1-\frac{\chi_2}{\chi_1} \geq -\mu_\Gamma^+ + \delta$ for $r$ sufficiently large for some positive constants $\ep$ and $\delta$. These conditions are precisely the conditions stated in \eqref{82}. Completeness of $g_0$ follows from the fact that each $N_i$ is closed, thus the $r$-coordinate along any divergent curve in $M_i$ must tend to infinity, and hence the length of such a curve is infinite.
\end{proof}

In particular, one may take
\begin{align}\label{70}
\Phi_i(r) = \begin{cases}
\sinh(r) & \text{if }k_i = 1 \\
e^r & \text{if }k_i = 0 \\
\cosh(r) & \text{if }k_i = -1.
\end{cases}
\end{align}
It is easy to verify in each of these cases that $\chi_1 = \frac{1}{2}$ and $\chi_2 = 0$, and thus the eigenvalues of $-g_0^{-1}A_{g_0}$ in these cases are $(\frac{1}{2},\dots,\frac{1}{2})$. Indeed, the three cases in \eqref{70} correspond to locally hyperbolic ends. To construct examples when $\mu_\Gamma^+>1$ which are not locally hyperbolic, one can take $k=-1$ and $\Phi(r) = e^{\alpha \sin r}$ for suitable $\alpha>0$. Indeed, in this case
\begin{align*}
\frac{\Phi'(r)^2 - k}{2\Phi(r)^2} = \frac{\alpha^2\cos^2(r)}{2} + \frac{1}{2e^{2\alpha\sin r}} \geq \frac{1}{2e^{2\alpha\sin r}} \geq \ep(\alpha)>0
\end{align*}
and
\begin{align}\label{84}
\frac{2\Phi(r)\Phi''(r)}{\Phi'(r)^2 - k} = \frac{-2\alpha\sin r + 2\alpha^2 \cos^2 r}{\alpha^2\cos^2 r + e^{-2\alpha\sin r}} \geq -\frac{2\alpha\sin r}{\alpha^2\cos^2 r + e^{-2\alpha\sin r}}.
\end{align}
Since the quantity on the RHS of \eqref{84} tends to zero uniformly on $\mathbb{R}$ as $\alpha\rightarrow 0$, it can be made greater than $-\mu_\Gamma^+ + 1 + \delta$ if $\mu_\Gamma^+>1$ and both $\delta$ and $\alpha$ are sufficiently small.

\section{Existence results in the positive case}\label{pos}

In this section we prove Theorem \ref{A'} and Corollary \ref{A}. As in the proof of Theorem \ref{H}, our proofs of these results use an exhaustion argument, and we will make use of results analogous to Theorems \ref{AA} and \ref{BB}. The first of these is an existence result of Guan \cite{Gua07}:

\begin{customthm}{C}[\cite{Gua07}]\label{EE}
	\textit{Let $(N,g_0)$ be a smooth compact Riemannian manifold of dimension $n\geq 3$ with non-empty boundary $\partial N$, and suppose that $(f,\Gamma)$ satisfies \eqref{21'}--\eqref{24'}. Let $\psi\in C^\infty(N\times\mathbb{R})$ be positive and $\xi\in C^\infty(\partial N)$, and suppose there exists $\bar{u}\in C^\infty(N)$ such that 
		\begin{align*}
		\begin{cases}
		f(\lambda(g_{\bar{u}}^{-1}A_{g_{\bar{u}}})) \geq \psi(\cdot, \bar{u}) & \text{in }N\backslash\partial N \\
		\bar{u} = \xi & \text{on }\partial N. 
		\end{cases}
		\end{align*}
		Then there exists a solution $u\in C^\infty(N)$ to 
		\begin{align}\label{25}
		\begin{cases}
		f(\lambda(g_u^{-1}A_{g_u})) = \psi(\cdot, u) & \text{in }N\backslash\partial N \\
		u = \xi & \text{on }\partial N
		\end{cases}
		\end{align}
		satisfying $u\leq \bar{u}$ in $N$. }
\end{customthm}

Whilst Guan obtains global estimates on solutions to \eqref{25}, in keeping with the proof of Theorem \ref{H} we will use local estimates depending only on an upper bound for the solution. Such estimates have been considered by various authors, see e.g.~\cite{GW03b, LL03, Che05, JLL07, Li09, Wang09}. We will use them in the following form, obtained e.g.~in \cite{Che05}: 

\begin{customthm}{D}[\cite{Che05}]\label{FF}
	\textit{Let $(N,g_0)$ be a smooth Riemannian manifold of dimension $n\geq 3$, possibly with non-empty boundary, and suppose that $(f,\Gamma)$ satisfies \eqref{21'}--\eqref{24'}. Let $\psi\in C^\infty(N)$ be positive and suppose $g_u = e^{2u}g_0$,  $u\in C^4(B_r)$, satisfies 
	\begin{align*}
	f(\lambda(g_u^{-1}A_{g_u})) = \psi, \quad \lambda(g_u^{-1}A_{g_u}) \in\Gamma 
	\end{align*}
in a geodesic ball $B_r$ contained in the interior of $N$. Then 
\begin{align*}
|\nabla_{g_0} u|_{g_0}^2(x) + |\nabla_{g_0}^2 u|_{g_0}(x) \leq C(r^{-2} + e^{2\sup_{B_r}u}) \quad \text{for }x\in B_{r/2},
\end{align*}
where $C$ is a constant depending on $n, f, \Gamma, \|g_0\|_{C^4(B_r)}$ and  $\|\psi\|_{C^2(B_r)}$ but independent of $\inf_{B_r} \psi$.}
\end{customthm}

We now give the proof of Theorem \ref{A'}:

\begin{proof}[Proof of Theorem \ref{A'}]
	Let us denote $\Lambda = \sup_M v <\infty$ and define the metric $g_\Lambda = e^{2\Lambda}g_0$. Since we assume $\lambda(g_0^{-1}A_{g_0})\in\Gamma$ on $M$, it holds that $\lambda(g_\Lambda^{-1}A_{g_\Lambda})\in\Gamma$ on $M$ with
	\begin{align*}
	f(\lambda(g_{\Lambda}^{-1}A_{g_\Lambda})) = e^{-2\Lambda}f(\lambda(g_0^{-1}A_{g_0}))>0. 
	\end{align*}
	Let $\{M_j\}_{j=1}^\infty$ be an exhaustion of $M$ by compact manifolds with boundary, and for each $j$ define
	\begin{equation*}
	\ep_{j} = \min_{M_j} f(\lambda(g_{\Lambda}^{-1}A_{g_{\Lambda}}))>0.
	\end{equation*}
	Let $\ep = \lim_{j\rightarrow\infty} \ep_j \geq 0$, which is well-defined since $\{\ep_j\}$ is a non-increasing sequence of positive numbers. We first consider Case 1, where we assume $\inf_M f(\lambda(g_0^{-1}A_{g_0}))=0$, or equivalently $\ep=0$. Since
	\begin{align*}
	f(\lambda(g_{{\Lambda}}^{-1}A_{g_{{\Lambda}}})) \geq  \ep_{j} \quad \text{in }M_j,
	\end{align*}
	Theorem \ref{EE} implies the existence of a smooth solution $g_{u_j} = e^{2u_j}g_0$ to
	\begin{equation}\label{6'}
	\begin{cases}
	f(\lambda(g_{u_{j}}^{-1}A_{g_{u_{j}}}))= \ep_{j}, \quad \lambda(g_{u_{j}}^{-1}A_{g_{u_{j}}})\in\Gamma & \text{in }M_j\backslash\partial M_j \\
	u_{j} = \Lambda & \text{on }\partial M_j
	\end{cases}
	\end{equation}
	satisfying $u_{j} \leq \Lambda$ in $M_j$.
	
	Now, since $R_{g_{u_{j}}}>0$, the function $w_j = e^{\frac{n-2}{2}u_j}$ satisfies 
	\begin{align*}
	\begin{cases}
	-\Delta_{g_0} w_{j} + \frac{n-2}{4(n-1)}R_{g_0} w_{j} > 0 & \text{in }M_j\backslash\partial M_j \\
	w_{j} = e^{\frac{n-2}{2}\Lambda} \geq e^{\frac{n-2}{2}v} & \text{on }\partial M_j. 
	\end{cases}
	\end{align*}
On the other hand, since $g_v$ is assumed to have nonpositive scalar curvature, it holds that $\widetilde{v} = e^{\frac{n-2}{2}v}$ satisfies
\begin{align*}
-\Delta_{g_0}\widetilde{v} + \frac{n-2}{4(n-1)}R_{g_0}\widetilde{v}\leq 0 \quad \text{on }M. 
\end{align*} 
The maximum principle therefore implies $w_{j} \geq \widetilde{v}$ in $M_j$, and hence
\begin{align*}
\widetilde{v} \leq w_j \leq e^{\frac{n-2}{2}\Lambda} \quad \text{in }M_j,
\end{align*}
or equivalently
	\begin{equation*}
	v \leq u_{j} \leq \Lambda \quad\text{in }M_j.
	\end{equation*}
	
	We now appeal to Theorem \ref{FF}, which implies that for any compact $K\subset M^n$ and $j$ sufficiently large so that $K\subset M_j\backslash \partial M_j$, we have
	\begin{align}\label{70'}
	\| u_{j}\|_{C^{2}(K)} \leq C_{K}.
	\end{align}
	We can therefore extract a subsequence which converges in $C^{1,\alpha}_{\operatorname{loc}}(M^n)$ to some $u\in C^{1,1}_{\operatorname{loc}}(M^n)$ satisfying
	\begin{align}\label{12'}
	\begin{cases}
	\lambda(g_{u}^{-1}A_{g_{u}})\in\partial \Gamma & \text{a.e.~on }M \\
	v \leq u \leq \Lambda & \text{on }M. 
	\end{cases}
	\end{align}
	Since we assume that $g_v$ is complete, the lower bound $u \geq v$ in \eqref{12'} implies that $g_u$ is also complete, as required. 
	
	Next we consider Case 2, where we assume $\inf_M f(\lambda(g_0^{-1}A_{g_0}))>0$, or equivalently $\ep>0$. Then we may run an argument almost identical to the above, except that we replace the RHS of the equation in \eqref{6'} with the function $\widetilde{\ep}\psi$ (independently of $j$), where $\widetilde{\ep}$ is such that 
	\begin{align*}
	f(\lambda(g_{\Lambda}^{-1}A_{g_{\Lambda}})) \geq \widetilde{\ep}\sup_M \psi \quad \text{on }M. 
	\end{align*}
	One then obtains the estimate \eqref{70'} in the same way as before, from which local uniform ellipticity follows. The theory of Evans-Krylov \cite{Ev82, Kry82} then implies a local $C^{2,\alpha}$ estimate, and finally classical Schauder theory (see e.g.~\cite{GT}) yields higher order estimates. One may consequently extract a subsequence which converges in $C^\infty_{\operatorname{loc}}(M^n)$ to some $\widetilde{u}\in C^\infty_{\operatorname{loc}}(M^n)$ satisfying 
	\begin{align*}
	\begin{cases}
	f(\lambda(g_{\widetilde{u}}^{-1}A_{g_{\widetilde{u}}})) = \widetilde{\ep}\psi, \quad \lambda(g_{\widetilde{u}}^{-1}A_{g_{\widetilde{u}}})\in\Gamma & \text{on }M \\
	v \leq \widetilde{u} \leq \Lambda & \text{on }M. 
	\end{cases}
	\end{align*}
	It is then easy to see that $u\defeq \widetilde{u} + \frac{1}{2}\ln\widetilde{\ep}$ satisfies \eqref{81}, and completeness follows from the same reasoning as in Case 1. 
\end{proof}

We now turn to Corollary \ref{A}, for which we will require some additional definitions and notation. Suppose that $M^n$ is a smooth non-compact manifold for which there exists a compact set $K\subset M^n$, a number $R>0$ and a diffeomorphism $\Phi: M^n\backslash K\rightarrow \mathbb{R}^n\backslash B_R(0)$. Let $r\geq 1$ be a smooth function on $M^n$ which agrees with the Euclidean distance $|x|$ to the origin under the identification $\Phi$ in a neighbourhood of infinity, and let $\hat{g}$ be a smooth Riemannian metric on $M^n$ which agrees with the Euclidean metric under the identification $\Phi$ in a neighbourhood of infinity. For $\tau>0$, a Riemannian metric $g$ on $M^n$ is then called $W^{k,p}_{-\tau}$ (resp.~$C^k_{-\tau}$) asymptotically flat (with one end) if 
\begin{align*}
g-\hat{g} \in W^{k,p}_{-\tau}(M^n) \quad (\text{resp. }C^k_{-\tau}(M^n)).
\end{align*} 
Here, we say a tensor $T$ belongs to $W^{k,p}_{-\tau}(M)$ if it belongs to $L^1_{\operatorname{loc}}$ and
\begin{align*}
\|T\|_{W^{k,p}_{-\tau}(M)}\defeq \sum_{j=0}^k \big\|r^{\tau - \frac{n}{p}-j}\nabla^j T\big\|_{L^p(M)}<\infty,
\end{align*}
and we say that $T$ belongs to $C^k_{-\tau}$ if 
\begin{align*}
\|u\|_{C^k_{-\tau}(M)} \defeq \sum_{j=0}^k \sup_M \big(r^{\tau+j}|\nabla^j u|\big)<\infty. 
\end{align*}
There is also the notion of weighted H\"older spaces, although we will not need them here. The definitions above extend to asymptotically flat manifolds with more than one end in the obvious way. 

In the non-compact setting, the Yamabe invariant $Y(M^n,[g_0])$ is defined analogously to the compact setting, taking an infimum over compactly supported functions:

\begin{align*}
Y(M^n,[g_0]) = \inf_{0\not = v\in C_c^\infty(M^n)}\frac{\int_M (\frac{4(n-1)}{n-2}|\nabla_{g_0} v|_{g_0}^2 + R_{g_0} v^2) \,dv_{g_0}}{(\int_M v^{\frac{2n}{n-2}}\,dv_{g_0})^{\frac{n-2}{n}}}.
\end{align*}
It was observed by Maxwell in \cite{Max05} (see the discussion at the start of Section 5 therein) that if $(M^n,g_0)$ is a $W^{k,p}_{-\tau}$ asymptotically flat manifold with $\tau\in(0,n-2)$, $k\geq 2$ and $k>\frac{n}{p}$, then $Y(M^n,[g_0])>0$ if $[g_0]$ contains a metric $g$ with $R_g\geq 0$. Moreover, $Y(M^n,[g_0])>0$ if and only if there exists a positive function $\phi$ with $\phi - 1 \in W^{k,p}_{-\tau}(M^n)$ such that $g = \phi^{\frac{4}{n-2}}g_0$ is scalar flat -- see \cite[Proposition 3]{Max05}

We now give the proof of Corollary \ref{A}; the proof is largely the same as the proof of Theorem \ref{A'}, except a little more care is needed to obtain a $C^0_{-\tau}$-asymptotically flat solution.

\begin{proof}[Proof of Corollary \ref{A}]
In light of the discussion above, the hypotheses of Corollary \ref{A} imply that there exists a smooth, $W^{2,p}_{-\tau}$-asymptotically flat, scalar flat conformal metric $g= \Phi^{4/n-2}g_0$ with $\Phi - 1\in W^{2,p}_{-\tau}(M^n)$. Since $p>n/2$, the Morrey embedding theorem for weighted Sobolev spaces (see \cite[Theorem 1.2]{Bart86}) implies $\Phi - 1\in C^0_{-\tau}(M^n)$. Therefore, $\Phi$ satisfies
	\begin{align*}
	\begin{cases}
	-\Delta_{g_0} \Phi + \frac{n-2}{4(n-1)}R_{g_0} \Phi = 0 & \text{in }M \\
	|\Phi(x) - 1| = O(|x|^{-\tau}) & \text{as }|x|\rightarrow\infty.
	\end{cases}
	\end{align*}

	Let $\beta\in(1,2)$ be a constant and let $\{M_j\}_{j=1}^\infty$ be an exhaustion of $M$ by compact manifolds with boundary. Then since $|\Phi(x)-1|=O(|x|^{-\tau})$, there exists $j = j(\beta)>0$ such that $\beta > \Phi$ on $M\backslash M_j$. Therefore
	\begin{align*}
	\begin{cases}-\Delta_{g_0} \beta + \frac{n-2}{4(n-1)}R_{g_0} \beta > 0 & \text{in }M_j \\
	\beta \geq \Phi & \text{on }\partial M_j,
	\end{cases} 
	\end{align*}
	and maximum principle then implies $\beta \geq \Phi$ in $M_j$. Therefore
	\begin{align*}
	\Phi \leq \beta \quad\text{in }M. 
	\end{align*}
	
	We now run the argument in the proof of Case 1 in Theorem \ref{A'} with $\beta$ in place of $e^{\frac{n-2}{2}\Lambda}$, which ultimately yields a metric $g_{u_\beta} = u_\beta^{\frac{4}{n-2}}g_0$, $u_\beta\in C^{1,1}_{\operatorname{loc}}(M^n)$, satisfying
	\begin{align*}
	\begin{cases}
	\lambda(g_{u_\beta}^{-1}A_{g_{u_\beta}})\in\partial \Gamma & \text{a.e.~on }M \\
	\Phi \leq u_\beta \leq \beta & \text{on }M. 
	\end{cases}
	\end{align*}
	To complete the proof, it remains to observe that all estimates obtained in the above procedure are independent of $\beta$, and hence we may take $\beta\rightarrow 1$ along a suitable sequence to obtain a $C^{1,1}$ solution $g_u = u^{\frac{4}{n-2}}g_0$ with $\Phi \leq u \leq 1$ on $M$. Since $|\Phi-1| = O(|x|^{-\tau})$, $C^0_{-\tau}$-asymptotic flatness of $g_u$ then follows. 
\end{proof}

\section{Examples in the positive case}\label{ex}

In this final section we give a class of examples where Corollary \ref{A} applies. In \cite{GWW14}, the authors considered the following Schwarzchild-type metric associated to the $\sigma_k$-operators for $k<\frac{n}{2}$. Let $(\rho,\theta)$ denote the standard spherical coordinates on $M^n = (0,\infty)\times \mathbb{S}^{n-1} = R^n\backslash\{0\}$, $d\Theta^2$ the round metric on $\mathbb{S}^{n-1}$, $r$ the radial coordinate on $\mathbb{R}^n$ and $|dx|^2$ the Euclidean metric on $\mathbb{R}^n$. Then
\begin{align*}
g_{\operatorname{Sch}}^k \defeq \bigg(1-\frac{2m}{\rho^{\frac{n-2k}{k}}}\bigg)^{-1}d\rho^2 + \rho^2 d\Theta^2 = \bigg(1+\frac{m}{2r^{\frac{n-2k}{k}}}\bigg)^{\frac{4k}{n-2k}}|dx|^2
\end{align*}
satisfies
\begin{align*}
\lambda\big((g_{\operatorname{Sch}}^k)^{-1}A_{g_{\operatorname{Sch}}^k}\big)\in\partial\Gamma_k^+ \quad \text{on }M^n,
\end{align*}
as can be seen from the proof of Proposition \ref{76} below (which gives a slightly more general construction). Moreover, $g_{\operatorname{Sch}}^k$ is $W^{l,\infty}_{-\tau}$-asymptotically flat for all $l\in\mathbb{N}$ when $\tau = \frac{n-2k}{k}$, and $W^{l,p}_{-\tau}$-asymptotically flat for all $l\in\mathbb{N}$ and $p\in[1,\infty)$ when $\tau<\frac{n-2k}{k}$.  
When $k=1$, $g_{\operatorname{Sch}}^1$ is the standard Riemannian Schwarzchild metric and $m$ is the ADM mass. When $k\geq 2$, $g_{\operatorname{Sch}}^k$ corresponds to the Schwarzchild metric in the corresponding Lovelock theory of gravity \cite{Love71}, and $m$ is related to the so-called Gauss-Bonnet-Chern (GBC) mass. 

More generally, associated to any cone $\Gamma$ satisfying \eqref{21'}, \eqref{22'} and $\mu_\Gamma^+>1$, we define the following Schwarzchild-type metric: 
\begin{align*}
g_{\operatorname{Sch}}^{\Gamma} = \bigg(1+\frac{m}{2r^{\mu_\Gamma^+-1}}\bigg)^{\frac{4}{\mu_\Gamma^+-1}}|dx|^2,
\end{align*}
where $m\in\mathbb{R}$ is a parameter. We have: 
\begin{prop}\label{76}
	Suppose $\Gamma$ satisfies \eqref{21'}, \eqref{22'} and $\mu_\Gamma^+>1$. Then the metric $g_{\operatorname{Sch}}^\Gamma$ satisfies
	\begin{align*}
	\lambda\big((g_{\operatorname{Sch}}^\Gamma)^{-1}A_{g_{\operatorname{Sch}}^\Gamma}\big)\in\partial\Gamma \quad \text{on }M^n. 
	\end{align*}
	Moreover, the scalar curvature of $g_{\operatorname{Sch}}^\Gamma$ is positive if $\Gamma\not=\Gamma_1^+$. 
\end{prop}

\begin{proof}
	If $g_v = v^{\frac{4}{n-2}}g_{\mathbb{R}^n}$ where $v=v(r)$, then a routine computation yields
	\begin{align*}
	g_{\mathbb{R}^n}^{-1}A_{g_v} = \chi_1\operatorname{Id} - \chi_2 \frac{x}{r}\otimes \frac{x}{r}
	\end{align*}
	where
	\begin{align*}
	\chi_1 = -\frac{2}{n-2}\frac{v_r}{vr} - \frac{2}{(n-2)^2}\frac{v_r^2}{v^2} \quad\text{and}\quad \chi_2 = \frac{2}{n-2}\bigg(\frac{v_{rr}}{v} - \frac{v_r}{vr}\bigg) -\frac{2n}{(n-2)^2}\frac{v_r^2}{v^2}. 
	\end{align*}
	In our case, $v = (1+\frac{m}{2r^{p}})^{\frac{n-2}{p}}$ where $p = \mu_\Gamma^+ -1$, and another routine computation therefore yields 
	\begin{align*}
	\chi_1 = \frac{m}{r^{p+2}}\bigg(1+\frac{m}{2r^p}\bigg)^{-1} - \frac{m^2}{2r^{2p+2}}\bigg(1+\frac{m}{2r^p}\bigg)^{-2} > 0 \quad \text{and} \quad \chi_2 = (p+2)\chi_1. 
	\end{align*}
	Now, $\lambda(g_{\mathbb{R}^n}^{-1}A_{g_v}) = (\chi_1-\chi_2, \chi_1, \dots, \chi_1)$, and by definition of $\mu_\Gamma^+$, this vector belongs to $\partial\Gamma$ if and only if $\chi_1 - \chi_2 = -\mu_\Gamma^+\chi_1$, which is clearly true by definition of $p$. 
	
	Finally, to see that the scalar curvature of $g_{\operatorname{Sch}}^\Gamma$ is positive when $\Gamma\not=\Gamma_1^+$, we compute
	\begin{align*}
	\sigma_1(g_v^{-1}A_{g_v}) & = \bigg(1+\frac{m}{2r^p}\bigg)^{-4/p}\sigma_1(g_{\mathbb{R}^n}^{-1}A_{g_v}) = \bigg(1+\frac{m}{2r^p}\bigg)^{-4/p}(n-1-\mu_\Gamma^+)\chi_1.
	\end{align*}
	The claim then follows since $\Gamma\not=\Gamma_1^+$ is equivalent to $\mu_\Gamma^+<n-1$, and we have already seen that $\chi_1>0$.
\end{proof}

It follows from Proposition \ref{76} that if $\Gamma\not=\Gamma_1^+$ and $\Gamma'\supset \Gamma$ is a cone satisfying \eqref{21'}, \eqref{22'} and $\overline{\Gamma'}\cap\overline\Gamma=\{0\}$, then
\begin{align*}
\lambda\big((g_{\operatorname{Sch}}^\Gamma)^{-1}A_{g_{\operatorname{Sch}}^\Gamma}\big)\in\Gamma'\quad \text{on }M^n. 
\end{align*}
Moreover, $g_{\operatorname{Sch}}^\Gamma$ is $W^{l,p}_{-\tau}$-asymptotically flat for all $l\in\mathbb{N}$ and $p<\infty$ when $\tau<\mu_{\Gamma}^+-1$. Thus, $g_{\operatorname{Sch}}^\Gamma$ satisfies the assumptions of Corollary \ref{A} for such a cone $\Gamma'$. Now, the existence of a metric $g\in[g_{\operatorname{Sch}}^\Gamma]$ satisfying 
\begin{align*}
\lambda(g^{-1}A_g)\in\partial\Gamma' \quad \text{on }M^n
\end{align*}
is trivial in this case, since one can simply take $g = g_{\operatorname{Sch}}^{\Gamma'}$. However, one can obtain non-trivial examples as follows. For any given compact subset $K\subset M$, there exists $\ep>0$ depending only on $K$ and $g_{\operatorname{Sch}}^\Gamma$ such that the following holds: if $h$ is a smooth $(0,2)$-tensor, compactly supported in $K$ and satisfying $\|h\|_{C^2(K)} \leq \ep$, then $g^h \defeq g_{\operatorname{Sch}}^\Gamma+h$ is a Riemannian metric on $M$ satisfying $\lambda((g^h)^{-1}A_{g^h})\in\Gamma'$ on $M$. For a general perturbation $h$, $[g_{\operatorname{Sch}}^\Gamma] \not=[g^h]$ and thus Corollary \ref{A} yields a non-trivial existence result in the conformal class $[g^h]$.

\renewcommand{\baselinestretch}{0}
\footnotesize
\bibliography{references}{}

\begin{thebibliography}{10}

\bibitem{AILA18}
{\sc P.~T. Allen, J.~Isenberg, J.~M. Lee, and I.~S. Allen}, {\em Weakly
  asymptotically hyperbolic manifolds}, Comm. Anal. Geom., 26 (2018),
  pp.~1--61.

\bibitem{ACF92}
{\sc L.~Andersson, P.~T. Chru\'{s}ciel, and H.~Friedrich}, {\em On the
  regularity of solutions to the {Y}amabe equation and the existence of smooth
  hyperboloidal initial data for {E}instein's field equations}, Comm. Math.
  Phys., 149 (1992), pp.~587--612.

\bibitem{Aub76}
{\sc T.~Aubin}, {\em \'{E}quations diff\'{e}rentielles non lin\'{e}aires et
  probl\`eme de {Y}amabe concernant la courbure scalaire}, J. Math. Pures Appl.
  (9), 55 (1976), pp.~269--296.

\bibitem{AM88}
{\sc P.~Aviles and R.~C. McOwen}, {\em Complete conformal metrics with negative
  scalar curvature in compact {R}iemannian manifolds}, Duke Math. J., 56
  (1988), pp.~395--398.

\bibitem{AM88-2}
\leavevmode\vrule height 2pt depth -1.6pt width 23pt, {\em Conformal
  deformation to constant negative scalar curvature on noncompact {R}iemannian
  manifolds}, J. Differential Geom., 27 (1988), pp.~225--239.

\bibitem{Bart86}
{\sc R.~Bartnik}, {\em The mass of an asymptotically flat manifold}, Comm. Pure
  Appl. Math., 39 (1986), pp.~661--693.

\bibitem{Bes87}
{\sc A.~L. Besse}, {\em Einstein manifolds}, vol.~10 of Ergebnisse der
  Mathematik und ihrer Grenzgebiete (3) [Results in Mathematics and Related
  Areas (3)], Springer-Verlag, Berlin, 1987.

\bibitem{CD10}
{\sc G.~Catino and Z.~Djadli}, {\em Conformal deformations of integral pinched
  3-manifolds}, Adv. Math., 223 (2010), pp.~393--404.

\bibitem{CGY02b}
{\sc S.-Y.~A. Chang, M.~J. Gursky, and P.~C. Yang}, {\em An a priori estimate
  for a fully nonlinear equation on four-manifolds}, J. Anal. Math., 87 (2002),
  pp.~151--186.

\bibitem{CGY02a}
\leavevmode\vrule height 2pt depth -1.6pt width 23pt, {\em An equation of
  {M}onge-{A}mp\`ere type in conformal geometry, and four-manifolds of positive
  {R}icci curvature}, Ann. of Math. (2), 155 (2002), pp.~709--787.

\bibitem{Che05}
{\sc S.-y.~S. Chen}, {\em Local estimates for some fully nonlinear elliptic
  equations}, Int. Math. Res. Not.,  (2005), pp.~3403--3425.

\bibitem{CLL23}
{\sc B.~Z. Chu, Y.~Y. Li, and Z.~Li}, {\em Liouville theorems for conformally
  invariant fully nonlinear equations {I}}, arXiv:2311.07542 [math.AP],
  (2023).

\bibitem{CLL23-2}
\leavevmode\vrule height 2pt depth -1.6pt width 23pt, {\em Private
  communication}, In preparation,  (2024).

\bibitem{DN22}
{\sc J.~A.~J. Duncan and L.~Nguyen}, {\em Differential inclusions for the
  {S}chouten tensor and nonlinear eigenvalue problems in conformal geometry},
  Adv. Math., 432 (2023), p.~Paper No. 109263.

\bibitem{DN23}
\leavevmode\vrule height 2pt depth -1.6pt width 23pt, {\em The
  $\sigma_k$-{L}oewner-{N}irenberg problem on {R}iemannian manifolds for
  $k<\frac{n}{2}$}, arxiv.org/abs/2310.01346,  (2023).

\bibitem{Evans82}
{\sc L.~C. Evans}, {\em Classical solutions of fully nonlinear, convex,
  second-order elliptic equations}, Comm. Pure Appl. Math., 35 (1982),
  pp.~333--363.

\bibitem{Ev82}
\leavevmode\vrule height 2pt depth -1.6pt width 23pt, {\em Classical solutions
  of fully nonlinear, convex, second-order elliptic equations}, Comm. Pure
  Appl. Math., 35 (1982), pp.~333--363.

\bibitem{FSY20}
{\sc J.~Fu, W.~Sheng, and L.~Yuan}, {\em Prescribed {$k$}-curvature problems on
  complete noncompact {R}iemannian manifolds}, Int. Math. Res. Not. IMRN,
  (2020), pp.~9559--9592.

\bibitem{GLW10}
{\sc Y.~Ge, C.-S. Lin, and G.~Wang}, {\em On the {$\sigma_2$}-scalar
  curvature}, J. Differential Geom., 84 (2010), pp.~45--86.

\bibitem{GW06}
{\sc Y.~Ge and G.~Wang}, {\em On a fully nonlinear {Y}amabe problem}, Ann. Sci.
  \'{E}cole Norm. Sup. (4), 39 (2006), pp.~569--598.

\bibitem{GWW14b}
{\sc Y.~Ge, G.~Wang, and J.~Wu}, {\em The {G}auss-{B}onnet-{C}hern mass of
  conformally flat manifolds}, Int. Math. Res. Not. IMRN,  (2014),
  pp.~4855--4878.

\bibitem{GWW14}
\leavevmode\vrule height 2pt depth -1.6pt width 23pt, {\em A new mass for
  asymptotically flat manifolds}, Adv. Math., 266 (2014), pp.~84--119.

\bibitem{GT}
{\sc D.~Gilbarg and N.~S. Trudinger}, {\em Elliptic partial differential
  equations of second order}, Classics in Mathematics, Springer-Verlag, Berlin,
  2001.

\bibitem{Gua07}
{\sc B.~Guan}, {\em Conformal metrics with prescribed curvature functions on
  manifolds with boundary}, Amer. J. Math., 129 (2007), pp.~915--942.

\bibitem{Guan08}
\leavevmode\vrule height 2pt depth -1.6pt width 23pt, {\em Complete conformal
  metrics of negative {R}icci curvature on compact manifolds with boundary},
  Int. Math. Res. Not. IMRN,  (2008).
\newblock Art. ID rnn 105, 25pp.

\bibitem{GVW03}
{\sc P.~Guan, J.~Viaclovsky, and G.~Wang}, {\em Some properties of the
  {S}chouten tensor and applications to conformal geometry}, Trans. Amer. Math.
  Soc., 355 (2003), pp.~925--933.

\bibitem{GW03a}
{\sc P.~Guan and G.~Wang}, {\em A fully nonlinear conformal flow on locally
  conformally flat manifolds}, J. Reine Angew. Math., 557 (2003), pp.~219--238.

\bibitem{GW03b}
\leavevmode\vrule height 2pt depth -1.6pt width 23pt, {\em Local estimates for
  a class of fully nonlinear equations arising from conformal geometry}, Int.
  Math. Res. Not.,  (2003), pp.~1413--1432.

\bibitem{GSW11}
{\sc M.~J. Gursky, J.~Streets, and M.~Warren}, {\em Existence of complete
  conformal metrics of negative {R}icci curvature on manifolds with boundary},
  Calc. Var. PDE, 41 (2011), pp.~21--43.

\bibitem{GV03b}
{\sc M.~J. Gursky and J.~A. Viaclovsky}, {\em Fully nonlinear equations on
  {R}iemannian manifolds with negative curvature}, Indiana Univ. Math. J., 52
  (2003), pp.~399--419.

\bibitem{GV07}
\leavevmode\vrule height 2pt depth -1.6pt width 23pt, {\em Prescribing
  symmetric functions of the eigenvalues of the {R}icci tensor}, Ann. of Math.
  (2), 166 (2007), pp.~475--531.

\bibitem{HN23}
{\sc J.~Hogg and L.~Nguyen}, {\em Existence and uniqueness for the non-compact
  {Y}amabe problem of negative curvature type}, to appear in {A}nalysis in
  {T}heory and {A}pplications, https://arxiv.org/pdf/2311.10623.pdf,  (2023).

\bibitem{JLL07}
{\sc Q.~Jin, A.~Li, and Y.~Y. Li}, {\em Estimates and existence results for a
  fully nonlinear {Y}amabe problem on manifolds with boundary}, Calc. Var. PDE,
  28 (2007), pp.~509--543.

\bibitem{Kho09}
{\sc S.~Khomrutai}, {\em Regularity of singular solutions to sigma(k)-{Y}amabe
  problems}, ProQuest LLC, Ann Arbor, MI, 2009.
\newblock Thesis (Ph.D.)--University of Notre Dame.

\bibitem{Kry82}
{\sc N.~V. Krylov}, {\em Boundedly inhomogeneous elliptic and parabolic
  equations}, Izv. Akad. Nauk SSSR Ser. Mat., 46 (1982), pp.~487--523, 670.

\bibitem{LL03}
{\sc A.~Li and Y.~Y. Li}, {\em On some conformally invariant fully nonlinear
  equations}, Comm. Pure Appl. Math., 56 (2003), pp.~1416--1464.

\bibitem{LL05}
\leavevmode\vrule height 2pt depth -1.6pt width 23pt, {\em On some conformally
  invariant fully nonlinear equations. {II}. {L}iouville, {H}arnack and
  {Y}amabe}, Acta Math., 195 (2005), pp.~117--154.

\bibitem{LNX22}
{\sc Y.~Li, L.~Nguyen, and J.~Xiong}, {\em Regularity of viscosity solutions of
  the {$\sigma_k$}-{L}oewner-{N}irenberg problem}, Proc. Lond. Math. Soc. (3),
  127 (2023), pp.~1--34.

\bibitem{Li09}
{\sc Y.~Y. Li}, {\em Local gradient estimates of solutions to some conformally
  invariant fully nonlinear equations}, Comm. Pure Appl. Math., 62 (2009),
  pp.~1293--1326.

\bibitem{LN13}
{\sc Y.~Y. Li and L.~Nguyen}, {\em A generalized mass involving higher order
  symmetric functions of the curvature tensor}, Ann. Henri Poincar\'{e}, 14
  (2013), pp.~1733--1746.

\bibitem{LN14}
\leavevmode\vrule height 2pt depth -1.6pt width 23pt, {\em A compactness
  theorem for a fully nonlinear {Y}amabe problem under a lower {R}icci
  curvature bound}, J. Funct. Anal., 266 (2014), pp.~3741--3771.

\bibitem{LN14b}
\leavevmode\vrule height 2pt depth -1.6pt width 23pt, {\em Harnack inequalities
  and {B}\^{o}cher-type theorems for conformally invariant, fully nonlinear
  degenerate elliptic equations}, Comm. Pure Appl. Math., 67 (2014),
  pp.~1843--1876.

\bibitem{LN20b}
\leavevmode\vrule height 2pt depth -1.6pt width 23pt, {\em Solutions to the
  {$\sigma_k$}-{L}oewner-{N}irenberg problem on annuli are locally {L}ipschitz
  and not differentiable}, J. Math. Study, 54 (2021), pp.~123--141.

\bibitem{LN20}
\leavevmode\vrule height 2pt depth -1.6pt width 23pt, {\em Existence and
  uniqueness of {G}reen's functions to nonlinear {Y}amabe problems}, Comm. Pure
  Appl. Math., 76 (2023), pp.~1554--1607.

\bibitem{Love71}
{\sc D.~Lovelock}, {\em The {E}instein tensor and its generalizations}, J.
  Mathematical Phys., 12 (1971), pp.~498--501.

\bibitem{Max05}
{\sc D.~Maxwell}, {\em Solutions of the {E}instein constraint equations with
  apparent horizon boundaries}, Comm. Math. Phys., 253 (2005), pp.~561--583.

\bibitem{Sch84}
{\sc R.~M. Schoen}, {\em Conformal deformation of a {R}iemannian metric to
  constant scalar curvature}, J. Differential Geom., 20 (1984), pp.~479--495.

\bibitem{STW07}
{\sc W.~Sheng, N.~S. Trudinger, and X.-J. Wang}, {\em The {Y}amabe problem for
  higher order curvatures}, J. Differential Geom., 77 (2007), pp.~515--553.

\bibitem{Tru68}
{\sc N.~S. Trudinger}, {\em Remarks concerning the conformal deformation of
  {R}iemannian structures on compact manifolds}, Ann. Scuola Norm. Sup. Pisa
  Cl. Sci. (3), 22 (1968), pp.~265--274.

\bibitem{Via00a}
{\sc J.~A. Viaclovsky}, {\em Conformal geometry, contact geometry, and the
  calculus of variations}, Duke Math. J., 101 (2000), pp.~283--316.

\bibitem{Wang09}
{\sc X.-J. Wang}, {\em The {$k$}-{H}essian equation}, Geometric analysis and
  {PDE}s, Lecture Notes in Math., 1977 (2009), pp.~177--252.

\bibitem{Yam60}
{\sc H.~Yamabe}, {\em On a deformation of {R}iemannian structures on compact
  manifolds}, Osaka Math. J., 12 (1960), pp.~21--37.

\bibitem{Yuan22}
{\sc R.~Yuan}, {\em The partial uniform ellipticity and prescribed problems on
  the conformal classes of complete metrics}, https://arxiv.org/abs/2203.13212,
   (2022).

\bibitem{Z18}
{\sc Q.~S. Zhang}, {\em Minimizers of the sharp log entropy on manifolds with
  non-negative {R}icci curvature and flatness}, Math. Res. Lett., 25 (2018),
  pp.~1673--1693.

\end{thebibliography}
\bibliographystyle{siam}

\end{document}